\RequirePackage[dvipsnames]{xcolor}

\documentclass[12pt, a4paper]{amsart}

\definecolor{BleuTresFonce}{rgb}{0.215, 0.215, 0.36}
\definecolor{EgyptianBlue}{rgb}{0.06, 0.2, 0.65}
\definecolor{DeepViolet}{rgb}{0.3, 0.0, 0.5}
\definecolor{RoyalPurple}{rgb}{0.47, 0.32, 0.66}
\definecolor{DarkViolet}{rgb}{0.0, 0.5, 0.0}

\usepackage{amssymb,amsmath,amsthm,bbold}
\usepackage[mathscr]{euscript}
\usepackage{geometry}
\usepackage{tikz}
\usetikzlibrary{cd, intersections, calc}
\usepackage{amscd}
\usepackage[all]{xy}

\usepackage[colorlinks,final,hyperindex]{hyperref}
\hypersetup{citecolor=DarkViolet, urlcolor=blue, linkcolor=EgyptianBlue}

\usepackage[font=small, labelfont=bf, margin=10pt]{caption}
\usepackage{csquotes}
\usepackage{mathtools}

\numberwithin{equation}{subsection}

\newtheorem*{conjecture*}{Conjecture}

\newtheorem{theorem}[equation]{Theorem}
\newtheorem*{theorem*}{Theorem}
\newtheorem{thm}[equation]{Theorem}
\newtheorem{defi}[equation]{Definition}
\newtheorem{lemma}[equation]{Lemma}
\newtheorem{prop}[equation]{Proposition}
\newtheorem{prop-def}[equation]{Proposition-Definition}

\newtheorem{corollary}[equation]{Corollary}

\theoremstyle{definition}

\newtheorem*{notation*}{Notations}
\newtheorem*{example*}{Example}
\newtheorem*{claim*}{Claim}

\newcommand{\qbinom}[3]{\genfrac{[}{]}{0pt}{}{#1}{#2}_{#3}}

\newcommand{\bP}{{\mathbb P}}

\title{Finite Version of the $q$-Analogue of de Finetti's Theorem}

\author[A.\,Dordzhiev]{Adyan Dordzhiev}
\email{aedordzhiev@gmail.com}
\begin{document} 
\begin{abstract}
Let $q \in (0,1)$. We formulate an asymptotic version of the $q$-analogue of de Finetti's theorem. Using the convex structure of the space of $q$-exchangeable probability measures, we show that the optimal rate of convergence is of order $q^n$.
\end{abstract}
\maketitle
\tableofcontents

\section{Introduction}
\noindent Let \( S(\infty) \) denote the group of permutations of the natural numbers that move only finitely many elements. A random sequence \( X_1, X_2, X_3, \ldots \) is \textit{exchangeable} if permuting finitely many indices does not change the law of the sequence. That is, for any finite permutation \( \sigma \in S(\infty) \),
\[
(X_1, X_2, X_3, \ldots) \overset{d}{=} (X_{\sigma(1)}, X_{\sigma(2)}, X_{\sigma(3)}, \ldots).
\] The celebrated de Finetti's theorem states that an infinite random $\{0,1\}$-valued  exchangeable sequence is a mixture of i.i.d.\ Bernoulli sequences. In other words, the space of exchangeable probability measures on \( \{0,1\}^\infty \) is isomorphic (as a convex set) to the space of all Borel probability measures on \( [0,1] \). The isomorphism is given by the following formula
\begin{equation}
\label{classicalversion}
\mathbb{P} (X_{1} = 1, \ldots, X_{k} = 1, X_{k+1} = 0, \ldots, X_{n} = 0) \coloneqq \int_{0}^{1} p^{k} (1-p)^{n-k} \mu(dp).
\end{equation}
De Finetti's theorem can be extended to more general settings~\cite{f8ab7edc-b800-37cb-843f-446c06ee989b}. The theorem can be proved by establishing a connection with the Hausdorff moment problem~\cite{feller1971introduction}. Another proof can be obtained by the moment method~\cite{kirsch2018elementaryprooffinettistheorem}. There is also an alternative approach based on harmonic functions on the Pascal graph~\cite{Borodin_Olshanski_2016}.

In \cite{gnedin2009qanalogue},  \cite{Gnedin_2010} a deformation of the concept of classical exchangeability was studied.  
\begin{defi}
For $q > 0$, a probability measure $\mathbb{P}$ on $\{0, 1\}^{\infty}$ is $q$-exchangeable if for any $\varepsilon_{1},\ldots,\varepsilon_{n} \in \{0, 1\}^{\infty}$ and elementary transposition $(i,i+1)$,
\begin{equation}\label{def:q_exch}
\mathbb{P}(\varepsilon_{1},\ldots,\varepsilon_{i-1},\varepsilon_{i+1},\varepsilon_{i},\varepsilon_{i+2},\ldots,\varepsilon_{n}) = q^{\varepsilon_{i}-\varepsilon_{i+1}} \mathbb{P}(\varepsilon_{1},\ldots,\varepsilon_{n}).
\end{equation}
\end{defi}
\noindent In other words, each additional inversion introduces an exponential penalty governed by the parameter $q$. For $q \in (0,1)$, a $q$-analogue of de Finetti's theorem for this type of probability measures has been established in~\cite{gnedin2009qanalogue}. See section~~\eqref{sec:q-exch} for a detailed discussion.

The infinite nature of the phase space \( \{0,1\}^\infty \) plays a crucial role in both formulations, see the introduction of~\cite{10.1214/aop/1176994663} for a counterexample. However, de Finetti's theorem can also be obtained as a limit of the finite version \( \{0,1\}^n \) as \( n \to \infty \). It was shown in~\cite{10.1214/aop/1176994663}  that this convergence occurs at an optimal rate of order $1/n$.
In this note, we obtain a finite version of the $q$-analogue of de Finetti's theorem, in the spirit of~\cite{10.1214/aop/1176994663}, with convergence at the sharp rate of order~$q^n$.

\textbf{Acknowledgments.}
I am deeply grateful to Grigori Olshanski for suggesting the problem and for his guidance throughout this work.
\section{Preliminaries}

\subsection{\texorpdfstring{$q$}{q}-Exchangeability} \label{sec:q-exch} Assume $q \in (0, 1)$. We use the standard notation for the $q$-integer, $q$-factorial, $q$-binomial coefficient, and  $q$-Pochhammer symbol, respectively,
\[[n] \coloneqq \frac{1-q^{n}}{1-q}, \;\; [n]! \coloneqq [1]\cdot[2], \ldots\cdot[n], \;\;\qbinom{n}{k}{} \coloneqq \dfrac{[n]!}{[k]![n-k]!}, \] 
\[(x;q)_{n} \coloneqq \prod_{i=0}^{n-1} (1-xq^{i}), \;\; 0\leqslant n \leqslant \infty,\]
where $(x;q)_{0} \coloneqq 1$. Since $q \in (0,1)$, the $q$-Pochhammer symbol is well-defined for $n = \infty$.

For a given finite sequence $\omega \in \{0,1\}^{n}$, consider the number of inversions in $\omega$
\begin{equation*}
    \mathrm{inv}(\omega) := \#\{(i, j) : 1 \leq i < j \leq n \ \text{and} \ \omega_{i} > \omega_{j}\}.
\end{equation*}
Denote by $C_{n,k} := \{\omega \in \{0,1\}^{n} : \sum_{i=1}^{n} \omega_{i} = k\}$ the set of all binary sequences of length $n$ containing exactly $k$ ones. Consider the sequence $s_{n,k} := (1,1,\ldots,1,0,0,\ldots,0)$ in $C_{n,k}$, which has ones in the first $k$ positions and zeros in the remaining $n-k$ positions. This sequence has the largest number of inversions in $C_{n,k}$. Each $q$-exchangeable measure $\mathbb{P}$ on $\{0,1\}^{n}$ is defined by the following equation
\begin{equation}\label{def:fin_q_ex}
\bP(\sigma \cdot \omega) = q^{\text{inv}(\omega) - \text{inv}(\sigma \cdot \omega)} \bP(\omega), \quad \omega \in \{0,1\}^{n}, \sigma \in S(n).
\end{equation}
In particular, each $q$-exchangeable measure on $\{0,1\}^{n}$ is determined by its values on the family of sequences $\{s_{n,k}\}_{n,k}$, since
\begin{equation}
\mathbb{P}(\omega) = q^{\mathrm{inv}(\omega)} \, \mathbb{P}(s_{n,k}), \quad \omega \in C_{n,k}.
\end{equation}
Note that equation~\eqref{def:fin_q_ex} can be extended to the case where $\omega \in \{0,1\}^{\infty}$ and $\sigma \in S(\infty)$. It is still equivalent to~\eqref{def:q_exch}, since the difference $\mathrm{inv}(\omega) - \mathrm{inv}(\sigma \cdot \omega)$ is finite whenever $\sigma \in S(\infty)$.

We now prove a useful property of the  function $\text{inv}(\omega)$.
\begin{prop}
\begin{equation}\label{inv_comput}
\sum_{\omega\in C_{n,k}} q^{\text{inv}(\omega)} = \qbinom{n}{k}{}.    
\end{equation}
\end{prop}
\begin{proof}
The $q$-binomial coefficient is uniquely determined by the following recurrence relation
\begin{equation*}
\qbinom{n}{k}{} = q^{k} \qbinom{n-1}{k}{} + \qbinom{n-1}{k-1}{}.
\end{equation*}
By forgetting the last entry in each sequence, we obtain the decomposition
\[
C_{n,k} = C_{n-1,k-1} \sqcup C_{n-1,k},
\]
and as a result, we have
\begin{equation*}
\sum_{\omega\in C_{n,k}} q^{\mathrm{inv}(\omega)} = q^{k}\sum_{\omega\in C_{n-1,k}} q^{\mathrm{inv}(\omega)} +  \sum_{\omega\in C_{n-1,k-1}} q^{\mathrm{inv}(\omega)}.
\end{equation*}
This identity coincides with the recurrence relation defining the $q$-binomial coefficient.

\end{proof}

\noindent In \cite{gnedin2009qanalogue}, a $q$-analogue of de Finetti's theorem was established. Consider the $q$-analogue of the interval $[0, 1]$ \[\Delta_{q} \coloneqq \{1, q, q^{2}, \ldots\} \cup \{0\}.\] 
 
\noindent
For each \( x \in \Delta_q \), we define a \( q \)-analogue of the Bernoulli measure $\nu^{q}_{x}$ on \( \{0,1\}^\infty \) and \( \{0,1\}^n \) as the unique $q$-exchangeable measure whose values on standard cylinder sets are assigned according to the formula
\begin{equation} \label{q-bernoulli}
    \nu^{q}_{x}(X_{1} = 1,\ldots,X_{k}=1,X_{k+1}=0,\ldots,X_{n}=0) \coloneqq q^{-k(n-k)} x^{n-k} (x; q^{-1})_{k}.
\end{equation}
Interpreting $x \in \Delta_{q}$ as the probability of a zero, the polynomial defined in ~\eqref{q-bernoulli} plays the role of a 
$q$-analogue for the binomial term $x^k (1 - x)^{n - k}$.
\begin{theorem}(Gnedin-Olshanski)\label{qversion_theorem}
$q$-exchangeable probability measures on $\{0, 1\}^{\infty}$ are in one-to-one correspondence with probability measures on $\Delta_{q}$. The bijection has the form
\begin{equation}
\label{qversion}
\mathbb{P}  \coloneqq \int_{\Delta_{q}}\nu^{q}_{x} \, \mu(dx),
\end{equation}
\end{theorem}
\noindent The classical version corresponds to the limit $q \to 1$. As $q$ increases, the set $\Delta_{q}$ becomes denser, and at $q =1$, it fills the entire interval $[0, 1]$.

\subsection{Finite form of classical version} We recall the main result from \cite{10.1214/aop/1176994663}. Given a probability measure $\mu$ on $[0,1]$, define a probability measure $\mathbb{P}_{\mu,n}$ on $\{0, 1\}^{n}$ as
\begin{equation}\label{finclintegral}
    \mathbb{P}_{\mu,n} (A) \coloneqq \int_{0}^{1} \nu_{p}(A) \mu(dp), \quad A \subset \{0,1\}^{n},
\end{equation}
 where $\nu_{p}$ denotes the Bernoulli measure on $\{0, 1\}^{n}.$ Recall that the map  
$\mu \mapsto \bP_{\mu,n}$
 is not surjective.

Let $\pi_k$ denote the canonical projection from $\{0,1\}^n$ onto its first $k$ coordinates, and let $\mathbb{P}_k$ denote the pushforward of $\mathbb{P}$ under $\pi_k$. Clearly, $(\mathbb{P}_{\mu,n})_k = \mathbb{P}_{\mu,k}$. The variational distance between two probability measures $\mu$ and $\nu$ on $(\Omega, \mathcal{F})$ is defined as
\begin{equation*}
\|\mu - \nu\| \coloneqq 2\, \sup\limits_{A \in \mathcal{F}} |\mu(A) - \nu(A)|.
\end{equation*}

\begin{thm}(Diaconis-Freedman)\label{classicalfinform}
Let $\mathbb{P}$ be an exchangeable measure on $\{0, 1\}^{n}$. Then there exists a probability measure $\mu$ on $[0, 1]$ such that 
\begin{equation}
    \|\mathbb{P}_{k}-\mathbb{P}_{\mu,k}\| \leqslant \frac{4k}{n}, \quad \text{for all } k \leqslant n,
\end{equation}
and this rate of convergence is sharp.
\end{thm}

\subsection{Extreme measures} The spaces of exchangeable and $q$-exchangeable probability measures are convex and compact; hence, by Choquet's theorem, they are the closed convex hulls of their extreme points. 

\begin{prop}
\begin{enumerate}
    \item Let $\Omega = \{0, 1\}^\infty$. The extreme points of the set of exchangeable probability measures on $\Omega$ are precisely the Bernoulli measures $\nu_p$, with $p \in [0, 1]$. For $q$-exchangeable measures, the extreme ones are parametrized by $x \in \Delta_q$ and are given by the measures $\nu^q_x$ defined in ~\eqref{q-bernoulli}.
    \item Let $\Omega = \{0, 1\}^n$. In this case, the sets of extreme exchangeable and $q$-exchangeable measures are finite. The extreme $q$-exchangeable measures, denoted by $e^q_0, e^q_1, \ldots, e^q_n$, are given by the formula

   \[ \begin{aligned}
        e^{q}_k (\omega) &=
        \begin{cases}
            q^{\text{inv}(\omega)}\dfrac{1}{\qbinom{n}{k}{}}, & \text{if } \omega \text{ contains } k \text{ ones and } n-k \text{ zeros}, \\
            0, & \text{otherwise.}
        \end{cases}
    \end{aligned}\]
    Setting $q = 1$, we obtain the extreme measures in the classical exchangeable case.
\end{enumerate}
\begin{proof}
Claim~(1) follows immediately from the bijections in~\eqref{classicalversion} and~\eqref{qversion}. For Claim~(2), note that due to $q$-exchangeability, the probability depends only on the number of ones, up to the scalar factor $q^{\operatorname{inv}(\omega)}$. This shows that each $e^q_k$ is extreme.
\end{proof}
\end{prop}
Fix \( n_1 \leqslant n \). For the extreme measure \( e^{q}_{n, n_1} \) and the measure \( \nu^{q}_{x} \) with parameter \( x = q^{n_1} \), we compute the probabilities in~\eqref{first_computation} and~\eqref{second_computation}, corresponding to the event that the first \( k \) entries of the sequence begin with exactly \( k_1 \) ones.

\begin{prop}\label{computation_prop}
\begin{equation}\label{first_computation}
(e^{q}_{n, n_1})_{k} \left(\text{s}_{k,k_{1}}\right) = q^{(n_{1} - k_{1})(k-k_{1})}\qbinom{n-k}{n_{1}-k_{1}}{} \bigg/ \qbinom{n}{n_{1}}{},
\end{equation}
\begin{equation}\label{second_computation}
(\nu^{q}_{x})_{k}  \left(\text{s}_{k,k_{1}}\right) = q^{(n_{1} - k_{1})(k-k_{1})}(q^{n_{1}}; q^{-1})_{k_{1}}.
\end{equation}
\end{prop}
\begin{proof}

We prove only~\eqref{first_computation}, the computation for~\eqref{second_computation} is analogous.
We have
\begin{equation*}
(e^{q}_{n, n_1})_{k} \left(\text{s}_{k,k_{1}}\right) = \sum_{\tilde{\omega} \in \{0,1\}^{n-k}} e^{q}_{n, n_1} \left(\text{s}_{k,k_{1}} \cup \tilde{\omega} \right),
\end{equation*}
where \( \text{s}_{k,k_{1}} \cup \tilde{\omega} \) denotes the concatenation of two sequences.
By counting inversions, we obtain
\begin{equation*}
\sum_{\tilde{\omega} \in \{0,1\}^{n-k}} e^{q}_{n, n_1} (1,\ldots,1,0,\ldots,0, \tilde{\omega}) 
=  \qbinom{n-k}{n_{1}-k_{1}}{}  e^{q}_{n, n_1} \left(\text{s}_{k, k_{1}} \cup \text{s}_{n-k, n_{1}-k_{1}}\right),
\end{equation*}
where \( \text{s}_{k,k_{1}} \cup \text{s}_{n-k,n_{1}-k_{1}} \) is the concatenation of two sequences of the same form.
The number of inversions in \( \text{s}_{k,k_{1}} \cup \text{s}_{n-k,n_{1}-k_{1}} \) equals \( (n_{1} - k_{1})(k - k_{1}) \). Therefore,
\begin{align*}
\qbinom{n-k}{n_{1}-k_{1}}{}  e^{q}_{n, n_1} \left(\text{s}_{k, k_{1}} \cup \text{s}_{n-k, n_{1}-k_{1}}\right)
&= q^{(n_{1}-k_{1})(k-k_{1})} \qbinom{n-k}{n_{1}-k_{1}}{} e^{q}_{n, n_1} \left(\text{s}_{n,n_{1}}\right) \\
&= q^{(n_{1}-k_{1})(k-k_{1})} \qbinom{n-k}{n_{1}-k_{1}}{} \bigg/ \qbinom{n}{n_{1}}{},
\end{align*}
which proves the claim.
\end{proof}
\section{Finite Form}
\subsection{Main result} In this section, we formulate an asymptotic version of Theorem~\eqref{qversion_theorem} in the sense of Theorem ~\eqref{classicalfinform}. Abusing notation, for a probability measure $\mu$ on $\Delta_q$, we denote by $\bP_{\mu,n}$ the probability measure given by
\[
\bP_{\mu,n}(A) = \int_{\Delta_q} \nu^{q}_{x}(A) \, \mu(dx), \quad A \subset \{0,1\}^{n},
\]
where $\nu^{q}_{x}$ denotes a probability measure on $\{0,1\}^{n}$ defined by~\eqref{q-bernoulli}.

\begin{theorem}\label{qfinform}
Let $\mathbb{P}$ be an $q$-exchangeable probability measure on $\{0, 1\}^{n}$. Then there exists a probability measure $\mu$ on $\Delta_{q}$ such that
    \begin{equation}
        \|\mathbb{P}_{k} - \mathbb{P}_{\mu,k}\| \leqslant c_{k} \cdot q^{n},  \quad \text{for all } k \leqslant n,
    \end{equation}
   where $c_k$ is a constant depending only on $k$.

\end{theorem}
\noindent The convergence rate of order~$q^{n}$ is sharp, as will be shown in Section~\ref{sec:sharp}. For convenience, we do not write the constant explicitly, only its existence is relevant for our purposes.
Using the convex structure of the space of $q$-exchangeable measures, we reduce the proof of Theorem ~\eqref{qfinform} to the case of an extreme measure.
\begin{lemma}\label{mainlemma} Fix $n_{1} \in \{0,1,\ldots,n\}$. In the notation of Theorem~\eqref{qfinform}, consider the extreme measure $\mathbb{P} = e^{q}_{n, n_{1}}$ and the probability measure $\mu = \delta_{q^{n_{1}}}$. Then
\begin{equation}\label{mainlemma_estimate}
\|\bP_{k} - \bP_{\mu,k}\| = \|(e^{q}_{n, n_{1}})_{k} -  (\nu^{q}_{q^{n_{1}}})_{k}\| \leqslant c_{k} \cdot q^{n}, \quad \text{for all } k \leqslant n,   
\end{equation}
where $c_k$ is a constant depending only on $k$.
\end{lemma}
\noindent Note that the estimate~\eqref{mainlemma_estimate} is uniform in the parameter $n_{1}$. The proof of the lemma is given in section ~\eqref{sec:extreme}. We now apply this lemma to prove the theorem.
\begin{proof}[Proof of Theorem ~\eqref{qfinform}]
Consider a convex decomposition of the measure $\bP$
\begin{equation*}
\bP = \alpha_{0}e^{q}_{0}+\ldots+\alpha_{n}e^{q}_{n}, \quad \sum_{i=0}^{n} \alpha_{i} = 1.
\end{equation*}
Then the corresponding pushforward measure is given by
\begin{equation*}
\bP_{k} = \alpha_{0}(e^{q}_{0})_{k}+\ldots+\alpha_{n}(e^{q}_{n})_{k}.
\end{equation*}
Define the probability measure $\mu$ by setting $\mu(q^{i}):=\alpha_{i}.$
Now consider the variation distance
\begin{equation*}
\|\bP_{k} - \bP_{\mu,k} \| = \|\sum_{i=0}^{n}\alpha_{i} (e^{q}_{n, i})_{k}  - \sum_{i=0}^{n}\alpha_{i} (\nu^{q}_{x})_{k} \| \leqslant \sum_{i=0}^{n} \alpha_{i} \|(e^{q}_{n, i})_{k} -  (\nu^{q}_{x})_{k}\|.
\end{equation*}
Finally, applying Lemma ~\eqref{mainlemma}, we obtain $\|\bP_{k} - \bP_{\mu,k}\| \leqslant c_{k} \cdot q^{n}$.
\end{proof}
\subsection{Extreme case}\label{sec:extreme} 
In this section, we prove Lemma~\eqref{mainlemma}. 
\begin{proof}
The variational distance~\eqref{mainlemma_estimate}  between the corresponding pushforward measures can be computed as follows
\begin{align*}
\big\| (e^{q}_{n, n_1})_k - (\nu^{q}_{q^{n_{1}}})_k \big\| 
&= \sum_{\omega \in \{0,1\}^{k}} \left| (e^{q}_{n, n_1})_k (\omega) - (\nu^{q}_{q^{n_{1}}})_k (\omega)\right| \\
&=  \sum_{k_1 = 0}^{k}\sum_{\omega \in C_{k,k_{1}}} q^{\text{inv}(\omega)}\left| (e^{q}_{n, n_1})_k (\text{s}_{k,k_{1}}) - (\nu^{q}_{q^{n_{1}}})_k (\text{s}_{k,k_{1}})\right| \\&= \sum_{k_1 = 0}^{k} \qbinom{k}{k_1}{}
\left| (e^{q}_{n, n_1})_k (\text{s}_{k,k_{1}}) - (\nu^{q}_{q^{n_{1}}})_k (\text{s}_{k,k_{1}})\right| \\&= \sum_{k_1 = 0}^{k} \qbinom{k}{k_1}{} \, q^{(n_1 - k_1)(k - k_1)}
\left| \qbinom{n-k}{n_{1}-k_{1}}{} \bigg/ \qbinom{n}{n_{1}}{}  
- (q^{n_1}; q^{-1})_{k_1} \right|.
\end{align*}
where the second identity follows from the $q$-exchangeability property, the third from Proposition~\eqref{inv_comput}, and the fourth from Proposition~\eqref{first_computation}. It follows that it suffices to analyse the expression 
\begin{equation}\label{expression2}
 q^{(n_{1} - k_{1})(k-k_{1})}\left|\qbinom{n-k}{n_{1}-k_{1}}{} \bigg/ \qbinom{n}{n_{1}}{} - (q^{n_{1}}; q^{-1})_{k_{1}}\right| =
\end{equation}
\begin{equation*}
 q^{(n_{1} - k_{1})(k-k_{1})} \left|\dfrac{[n-k]!}{[n]!} \dfrac{[n_{1}]!}{[n_{1}-k_{1}]!}\dfrac{[n-n_{1}]! }{[n-n_{1}-(k-k_{1})]!} - (q^{n_{1}}; q^{-1})_{k_{1}}\right|.
\end{equation*}
We consider two cases: $k_1 = k$ and $k_1 < k$.

\textbf{Case 1: \( k_1 = k \).}  
In this case, \eqref{expression2} reduces to
\begin{align} \label{expression3}
\left|\dfrac{[n-k]!}{[n]!} \dfrac{[n_{1}]!}{[n_{1}-k]!}\ - (q^{n_{1}}; q^{-1})_{k}\right|  
&= \left|\dfrac{ (q^{n_{1}}; q^{-1})_{k}}{\prod_{i=0}^{k-1}(1 - q^{n - i})} - (q^{n_{1}}; q^{-1})_{k}\right| \\
&=  \nonumber(q^{n_{1}}; q^{-1})_{k} \dfrac{1 - \prod_{i=0}^{k-1}(1 - q^{n - i})}{\prod_{i=0}^{k-1}(1 - q^{n - i})},
\end{align}
since $(q^{n_{1}}; q^{-1})_{k} \leqslant 1$ and
${\prod_{i=0}^{k-1}(1 - q^{n - i})} \geqslant (1 - q)^{k}$, we obtain the estimate
\begin{equation}
 (q^{n_{1}}; q^{-1})_{k} \dfrac{1 - \prod_{i=0}^{k-1}(1 - q^{n - i})}{\prod_{i=0}^{k-1}(1 - q^{n - i})} \leqslant \dfrac{1 - \prod_{i=0}^{k-1}(1 - q^{n - i})}{(1-q)^{k}},
\end{equation} 
applying the inequality $(1 - x)(1 - y) \geq 1 - x - y$ for $x, y \geqslant 0$, we get
\begin{equation}\label{express_k=k1}
\dfrac{1 - \prod_{i=0}^{k-1}(1 - q^{n - i})}{(1-q)^{k}} \leqslant \dfrac{\sum_{i=0}^{k-1} q^{n-i}}{(1-q)^{k}} = \dfrac{\sum_{i=0}^{k-1} q^{-i}}{(1-q)^{k}} q^{n} .
\end{equation} 
Hence, the upper bound for~\eqref{expression3} is proportional to~$q^n$, with the constant depending only on~$k$.

\textbf{Case 2: \( k_1 < k \).} In this case, expression~\eqref{expression2} can be rewritten as  
\begin{align}\label{expression4}
q^{(n_{1}-k_{1})(k-k_{1})}\left(q^{n_{1}}; q^{-1}\right)_{k_{1}} \left|\dfrac{\prod_{i=0}^{k-k_{1}-1}\left(1 - q^{n- n_{1} - i}\right) - \prod_{i=0}^{k-1}\left(1 - q^{n - i}\right)}{\prod_{i=0}^{k-1}\left(1 - q^{n - i}\right)}\right|.
\end{align}  
The expression~\eqref{expression4} depends on the sign of the difference in the numerator 
\begin{equation}
\left|\prod_{i=0}^{k-k_{1}-1}\left(1 - q^{n- n_{1} - i}\right) - \prod_{i=0}^{k-1}\left(1 - q^{n - i}\right)\right|.
\end{equation}  
If \(\prod_{i=0}^{k-k_{1}-1}\left(1 - q^{n- n_{1} - i}\right) - \prod_{i=0}^{k-1}\left(1 - q^{n - i}\right) \geqslant 0\), then we have  
\begin{equation}\label{expression_nonneg}
q^{(n_{1}-k_{1})(k-k_{1})}\left(q^{n_{1}}; q^{-1}\right)_{k_{1}} \dfrac{\prod_{i=0}^{k-k_{1}-1}\left(1 - q^{n- n_{1} - i}\right) - \prod_{i=0}^{k-1}\left(1 - q^{n - i}\right)}{\prod_{i=0}^{k-1}\left(1 - q^{n - i}\right)},
\end{equation}  
since \(q^{(n_{1}-k_{1})(k-k_{1})}\leqslant 1\), \(\prod_{i=0}^{k-k_{1}-1}\left(1 - q^{n- n_{1} - i}\right) \leqslant 1\) the upper bound for~\eqref{expression_nonneg} is  
\begin{equation*}
\left(q^{n_{1}}; q^{-1}\right)_{k_{1}} \dfrac{1 - \prod_{i=0}^{k-1}\left(1 - q^{n - i}\right)}{\prod_{i=0}^{k-1}\left(1 - q^{n - i}\right)},
\end{equation*}  
applying the same inequalities as in Case~1, we estimate the entire expression by ~\eqref{express_k=k1}
\begin{equation*}
\dfrac{\sum_{i=0}^{k-1} q^{-i}}{(1-q)^{k}} q^{n}. 
\end{equation*}
In the case when \(\prod_{i=0}^{k-k_{1}-1}\left(1 - q^{n- n_{1} - i}\right) - \prod_{i=0}^{k-1}\left(1 - q^{n - i}\right) < 0\), we write  
\begin{equation}
q^{(n_{1}-k_{1})(k-k_{1})}\left(q^{n_{1}}; q^{-1}\right)_{k_{1}} \dfrac{\prod_{i=0}^{k-1}\left(1 - q^{n - i}\right) - \prod_{i=0}^{k-k_{1}-1}\left(1 - q^{n- n_{1} - i}\right)}{\prod_{i=0}^{k-1}\left(1 - q^{n - i}\right)},
\end{equation}  
since \(\prod_{i=0}^{k-1}\left(1 - q^{n - i}\right) \leqslant 1\), the upper bound becomes  
\begin{equation}
q^{(n_{1}-k_{1})(k-k_{1})}\left(q^{n_{1}}; q^{-1}\right)_{k_{1}} \dfrac{1 - \prod_{i=0}^{k-k_{1}-1}\left(1 - q^{n - n_{1} - i}\right)}{\prod_{i=0}^{k-1}\left(1 - q^{n - i}\right)},
\end{equation}  
applying the same inequalities as in Case~1, we estimate the entire expression by  
\begin{equation}
q^{(n_1 - k_1)(k - k_1)}  
\frac{\sum_{i=0}^{k - k_1 - 1} q^{n - n_1 - i}}{(1 - q)^k} = \dfrac{\sum_{i=0}^{k-k_{1}-1} q^{n + n_{1}(k-k_{1}-1) + k_{1}(k_{1} -k) - i}}{(1-q)^{k}},
\end{equation}  
since \( k - k_1 - 1 \geqslant 0 \), we obtain a uniform bound with respect to the parameter \(n_{1}\)  
\begin{equation}
\dfrac{\sum_{i=0}^{k-k_{1}-1} q^{n + n_{1}(k-k_{1}-1) + k_{1}(k_{1} -k) - i}}{(1-q)^{k}} \leqslant \dfrac{\sum_{i=0}^{k-k_{1}-1} q^{k_{1}(k_{1} -k) - i}}{(1-q)^{k}} q^{n}.
\end{equation}
In each of the two cases, the upper bound is of order \( q^n \), with the constant depending only on \( k \) and \( k_1 \).

Combining the two cases, we conclude that the overall bound is of the form $c_k \cdot q^n$, where $c_k$ is a constant depending only on~$k$. This completes the proof of the lemma.

\end{proof}

\subsection{From finite to infinite}Since the set $\Delta_{q}$ is compact, the probability measures on  $\Delta_{q}$ are uniquely determined by sequences of their moments. Therefore, injectivity of the map ~\eqref{qversion} is automatic. Using Theorem ~\eqref{qfinform}, we prove the surjectivity of the map ~\eqref{qversion}, thereby rederiving the result of Gnedin--Olshanski ~\eqref{qversion}.
\begin{corollary}
The map ~\eqref{qversion} is surjective.
\begin{proof}
Let $\mathbb{P}$ be a $q$-exchangeable probability measure on $\{0,1\}^{\infty}$. 
Consider the natural projections $\mathbb{P}_n$ onto $\{0,1\}^n$. 
From Theorem~\eqref{qfinform} we obtain a family of measures $\mu_n$. 
By compactness of the space of probability measures on $[0,1]$, we can extract subsequence $\mu_{n_i}$ that converges weakly to a probability measure $\mu$. Consequently, we obtain the weak convergence $\mathbb{P}_{\mu_{n_i},k} \xrightarrow[n_i \to \infty]{\text{weakly}} \mathbb{P}_{\mu,k}.$
Since
$
\|\mathbb{P}_{\mu_{n_i},k} - \mathbb{P}_k\| \xrightarrow[n_i \to \infty]{} 0$, we have \(\bP_{\mu,k} = \bP_{k}\)  for all $k$.
 We conclude that $\mathbb{P} = \mathbb{P}_{\mu} = \displaystyle
 \int_{\Delta_{q}}\nu^{q}_{x} \,\mu(dx)$.

\end{proof}
\end{corollary}
\subsection{The rate is sharp}\label{sec:sharp}

We provide an example in which the lower bound for the variational distance in Theorem~\eqref{qfinform} is of order~$q^n$, confirming that this rate is optimal. The example is given by the extreme measure \( e^{q}_{n, n_1} \) and the measure \( \nu^{q}_{x} \) with parameter \( x = q^{n_1} \).
We begin by proving a technical lemma.

\begin{lemma}\label{example_lemma}
\begin{equation}
\dfrac{1 - \prod_{i=0}^{k-1}(1 - q^{n - i})}{\prod_{i=0}^{k-1}(1 - q^{n - i})} \geqslant \dfrac{q^{1-k}-q}{1-q} q^{n}.
\end{equation}
\end{lemma}

\begin{proof}
Since $\ln(1 - x) \leqslant -x$ for $x \in (0,1)$, we have
\begin{equation}
\prod_{i=0}^{k-1}(1 - q^{n - i}) \leqslant \exp\left(-\sum_{i=0}^{k-1}q^{n-i}\right),
\end{equation}
therefore,
\begin{equation}
\dfrac{1 - \prod_{i=0}^{k-1}(1 - q^{n - i})}{\prod_{i=0}^{k-1}(1 - q^{n - i})} = \dfrac{1}{\prod_{i=0}^{k-1}(1 - q^{n - i})} - 1 \geqslant \exp\left(\sum_{i=0}^{k-1}q^{n-i}\right) - 1.
\end{equation}
Since $\exp(x) - 1 \geqslant x$ for $x \geqslant 0$, it follows that
\begin{equation}
\dfrac{1 - \prod_{i=0}^{k-1}(1 - q^{n - i})}{\prod_{i=0}^{k-1}(1 - q^{n - i})} \geqslant \sum_{i=0}^{k-1}q^{n-i} = \dfrac{q^{1-k}-q}{1-q} q^{n}.
\end{equation}
\end{proof}
\begin{prop}
For $n_{1} \geqslant k$, we have 
\begin{equation}
\big\| (e^{q}_{n, n_1})_k - (\nu^{q}_{q^{n_{1}}})_k \big\| \geqslant  \tilde{c}_{k}\cdot q^{n},
\end{equation}
where $\tilde{c}_{k}$ is a constant depending only on $k$.
\end{prop}
\begin{proof}
As we have already shown, the variational distance between  \( e^{q}_{n, n_1} \) and \( \nu^{q}_{x} \) can be computed using the following formula
\begin{equation*}
\big\| (e^{q}_{n, n_1})_k - (\nu^{q}_{q^{n_{1}}})_k \big\| 
= \sum_{k_1 = 0}^{k} \qbinom{k}{k_1}{} \, q^{(n_1 - k_1)(k - k_1)}
\left| \qbinom{n-k}{n_{1}-k_{1}}{} \bigg/ \qbinom{n}{n_{1}}{}  
- (q^{n_1}; q^{-1})_{k_1} \right|.
\end{equation*}
Since $n_{1} \geqslant k$, we have $(q^{n_{1}}; q^{-1})_{k} \neq 0$. To obtain a lower bound, we consider only the term corresponding to $k_1 = k$ in the sum.
\begin{equation*}
\big\| (e^{q}_{n, n_1})_k - (\nu^{q}_{q^{n_{1}}})_k \big\|  \geqslant  (q^{n_{1}}; q^{-1})_{k}  \dfrac{1 - \prod_{i=0}^{k-1}(1 - q^{n - i})}{\prod_{i=0}^{k-1}(1 - q^{n - i})}.
\end{equation*}
Using the inequality $(q^{n_{1}}; q^{-1})_{k} \geqslant (1-q)^{k}$ and Lemma~\eqref{example_lemma}, we obtain
\begin{equation}
(q^{n_{1}}; q^{-1})_{k} \dfrac{1 - \prod_{i=0}^{k-1}(1 - q^{n - i})}{\prod_{i=0}^{k-1}(1 - q^{n - i})} \geqslant (1-q)^{k}  \dfrac{q^{1-k}-q}{1-q} q^{n}.
\end{equation}
Thus, we see that the lower bound is of order~$q^{n}$.
\end{proof}

\bibliographystyle{alpha}
\bibliography{bibliography}
\end{document}